\documentclass[a4paper,reqno,12pt]{amsart}

\usepackage[utf8]{inputenc}
\usepackage{amsmath, amssymb, amsthm, epsfig}
\usepackage{hyperref, latexsym}
\usepackage{url}
\usepackage[mathscr]{euscript}
\usepackage{color}
\usepackage{harpoon}
\usepackage{url}
\usepackage{mathabx}
\usepackage{tikz}

\usepackage{fullpage}
\usepackage{setspace}
\usepackage{adjustbox}

\usepackage{mathtools}
\mathtoolsset{showonlyrefs}

\def\today{\ifcase\month\or
  January\or February\or March\or April\or May\or June\or
  July\or August\or September\or October\or November\or December\fi
  \space\number\day, \number\year}

\newtheorem{theorem}{Theorem}
\newtheorem{conjecture}{Conjecture}
\newtheorem{lemma}[theorem]{Lemma}

\newtheorem{remark}[theorem]{Remark}
\newtheorem{definition}[theorem]{Definition}

\theoremstyle{definition}
\newtheorem{example}[theorem]{Example}

\newcommand{\I}{\mathcal{I}}

\newcommand{\n}{\mathbb{N}}
\newcommand{\z}{\mathbb{Z}}

\renewcommand{\r}{\mathbb{R}}

\newcommand{\re}{{\rm Re}\,}
\newcommand{\ft}{\widehat}

\begin{document}


\title[]{On a Sharp Fourier Extension Inequality on the Circle with Lacunary Spectrum}
\author[Gon\c{c}alves]{Felipe Gon\c{c}alves}
\date{\today}
\subjclass[2010]{}
\address{The University of Texas at Austin, 2515 Speedway, Austin, TX 78712, USA  \newline   {\&} IMPA - Instituto de Matemática Pura e Aplicada, Rio de Janeiro, 22460-320, Brazil.}
\email{felipe.ferreiragoncalves@austin.utexas.edu}
\email{goncalves@impa.br}
\author{João Paulo Ferreira}
\date{\today}
\subjclass[2010]{}
\address{IMPA - Instituto de Matemática Pura e Aplicada, Rio de Janeiro, 22460-320, Brazil.}
\email{joao.ferreira@impa.br}
\allowdisplaybreaks


\begin{abstract}
We prove a sharp Fourier extension inequality on the circle for the Tomas-Stein exponent for functions whose spectrum $\{\pm \lambda_n\}$ satisfies $\lambda_{n+1}>3 \lambda_{n}$.
\end{abstract}


\maketitle


\section{Introduction}

Fourier restriction theory has been established as one of the most relevant areas in modern harmonic analysis. The celebrated Tomas-Stein theorem lies at the heart of Fourier restriction theory; in the spherical case, its dual version states that
\begin{equation}\label{TomasSteinInequality}
	||\widehat{f\sigma}||_{L^{\frac{2(d+1)}{d-1}}(\r^d)}\leq C ||f||_{ L^2(\mathbb{S}^{d-1})},
\end{equation}
where $\sigma$ denotes the embedded surface measure on $\mathbb{S}^{d-1}$ and
$$\widehat{f\sigma}(x)=\int_{\mathbb{S}^{d-1}}f(w)e^{-i x \cdot w}d\sigma(w).$$
Questions concerning the best constant in \eqref{TomasSteinInequality} are quite more recent and have been proven to be surprisingly challenging. Foschi, in his celebrated paper \cite{Foschi2015}, proved that when $d=3$ constant functions are extremizers for \eqref{TomasSteinInequality} and all extremizers are of form $ce^{ix_0\cdot w}$ where $c$ is a constant. Dimension $d=3$ is the only dimension where the best constant is known for the sphere, although other related sharp extension inequalities have been investigated (see \cite{Biswas2025}, \cite{Carneiro2024a}, \cite{Carneiro2024}, \cite{Carneiro2015}, \cite{Carneiro2019}, \cite{Ciccone2023}, \cite{GonzalezRiquelme2024}). Sharp Tomas-Stein inequalities for the sphere and other conic sections can be tackled when the exponent is even, using the natural convolution structure that emerges from the problem. It is conjectured that constant functions are the only extremizers (modulo symmetries) for \eqref{TomasSteinInequality} for every $d\geq 2$. For $d=2$, in \cite{Carneiro2017}, the conjecture was proven to be \textit{locally} true, in the sense that constant functions are locally extremizers. Later in \cite{Goncalves2022}, Gonçalves and Negro proved that constant functions are local extremizers of \eqref{TomasSteinInequality} for $2\leq d \leq 60$, which adds an evidence for the conjecture.

The case $d=2$ is the last remaining even exponent coming from extension problems associated to conic sections (see \cite{Negro} and \cite{OliveiraeSilva2024}), hence it naturally attracts a lot of attention.  We let
$$C_{\text{opt}}:=\sup_{\stackrel{f\in L^2(\mathbb{S}^1)}{f\neq 0}}\dfrac{||\widehat{f\sigma}||_{L^{6}(\r^2)}^6}{||f||_{L^2(\mathbb{S}^1)}^6}.$$
As far as the $L^2\to L^6$ estimate in dimension $d=2$ is concerned, much partial progress has been made; it has been long established in \cite{Shao2016} that extremizers do exist. Moreover, they are even, can be taken as non-negative and smooth. We let $J_{n}$ denote the Bessel function of first kind, defined as $2\pi (-i)^nz^n J_n(|z|)=|z|^{n} \widehat{(z^n \sigma)}(z)$, where $z=x+iy$ with $x,y\in \r$.
\begin{conjecture}[Case $d=2$]\label{Conjecture}
	For all $f\in L^2(\mathbb{S}^1)\backslash\{0\}$,
	\begin{equation}\label{Inequalityofratios}
		\dfrac{||\widehat{f\sigma}||^6_{L^6(\r^2)}}{||f||^6_{L^2(\mathbb{S}^1)}}\leq C_{\text{opt}}=\dfrac{||\widehat{\sigma}||^6_{L^6(\r^2)}}{||\boldmath{\boldsymbol{1}}||^6_{L^2(\mathbb{S}^1)}}=(2\pi)^4\left(\int_0^{\infty}J_0^6(r)rdr\right).
	\end{equation}
\end{conjecture}

Ciccone and Gonçalves in \cite{Ciccone2024} proved Conjecture \ref{Conjecture} for the subspace of functions whose spectrum is under \textit{arithmetic constraints} (see Definition \ref{DefP(3)set} below for a precise definition), meaning that as long as the set $\text{spec}(f):=\{k\in \z;\,\widehat{f}(k)\neq 0\}$ is sufficiently sparse, the conjecture holds. The arithmetic constraints that they impose on $\text{spec}(f)$ is a generalizations of a \textit{Sidon set}. Classically, a Sidon set $A\subset \z$ is a set for which each element of $A+A$ has unique representation modulo permutation. Ciccone and Gonçalves in \cite{Ciccone2024} impose that if $A=\text{spec}(f)$ then each element $x\in A+A+A$ either has unique representation modulo permutation or is trivially represented as $x=x+a-a$ for some $a\in A \cap (-A)$. It turns out that the special case $\text{spec}(f)=\{\pm \lambda_{n}\}$, where $\{\lambda_n\}$ is a lacunary sequence with $\lambda_{n+1}/\lambda_{n}>q$ and $q\geq 5$ falls into these conditions. However, the spectrum such as $\{0\}\cup \{\pm 5^n\}_{n\geq 0}$ or $\{0\}\cup \{\pm 2^n\}_{n\geq 0}$ are not contemplated by their result. A natural question is:  

\emph{Can we prove Conjecture \ref{Conjecture} for $\text{spec}(f)\subset \{0\}\cup \{\pm q^n\}_{n\geq 0}$ for small $q>1$? More generally, can we prove Conjecture \ref{Conjecture} when $\text{spec}(f)$ is lacunary with small lacunary constant?}

\begin{theorem}\label{Lacunarycaseq>3}
	Let $f\in L^2(\mathbb{S}^1)$ be such that $\text{spec}(f)\subset A_{\lambda,q}:=\{\pm \lambda_n;n\geq 0\}$, where $\{\lambda_n\}$ is a lacunary sequence of non-negative integers satisfying $\lambda_{n+1}/\lambda_{n}>3$, and $\lambda_{0}:=0$. Then
	\begin{equation}\label{SharpInequality}
		||\widehat{f\sigma}||^6_{L^6(\r^2)}\leq (2\pi)^4\left(\int_0^{\infty}J_0^6(r)rdr\right)||f||^6_{L^2(\mathbb{S}^1)},
	\end{equation}
	and equality is attained if and only if $\widehat{f}(k)=0$ for every $k\neq 0$.
\end{theorem}
\subsection{Motivation, Notation and Strategy} 
The novelty in the proof of Theorem \ref{Lacunarycaseq>3} relies on the fact that, although lacunary sequences are not always in the conditions of \cite{Ciccone2024}, as long as the lacunary constant is large enough, the points that fail to satisfy these conditions can be classified, allowing a precise decomposition of the norm $||\widehat{f\sigma}||_{L^6(\r^2)}^6$ (see Lemma \ref{P3exceptions}).

Recall that a sequence $\{\lambda_{n}\}$ of non-negative integers is a \textit{lacunary sequence} if $\lambda_{n+1}/\lambda_{n}>q$ for some $q>1$. In this case, we say that $\{\lambda_n\}$ has \textit{lacunary constant $q$}.  To verify inequality \eqref{Inequalityofratios} for all $L^2$ functions, one could prove \eqref{Inequalityofratios} for all functions whose spectrum is lacunary. Similarly, one could prove  \eqref{Inequalityofratios} for all functions whose spectrum is finite, i.e., \textit{band-limited} functions. The band-limited case has been considered in \cite{OliveiraeSilva2022}, where they prove that if $\text{spec}(f)\subset [-30,30]$ then inequality \eqref{Inequalityofratios} holds, and later improved in \cite{Barker2023} for functions with $\text{spec}(f)\subset [-120,120]$. 

To prove Theorem \ref{Lacunarycaseq>3}, we use an explicit formula for $||\widehat{f\sigma}||_{L^6(\r^2)}^6$. Letting $A=\text{spec}(f)$ we have
\begin{equation}\label{FormulaEmanuelDiogo}
	\begin{array}{rl}
		(2\pi)^{-7}||\widehat{f\sigma}||_{L^6(\r^2)}^6&=\displaystyle{\sum_{D\in A^3}\sum_{\substack{n_1,\dots,n_6\in A\\n_1+n_2+n_3=D\\n_4+n_5+n_6=D}} \widehat{f}(n_1)\widehat{f}(n_2) \widehat{f}(n_3)\overline{\widehat{f}(n_4)}\overline{\widehat{f}(n_5)}\overline{\widehat{f}(n_6)}I(n_1,\dots,n_6)},
	\end{array}
\end{equation}
where throughout this paper we convention the notation $A^h:=\{n_1+\dots+n_h;\,n_1,\dots,n_h\in A\}$, and $I(n_1,\dots,n_6)$ denotes the integral
$$I(n_1,\dots,n_6)=\int_0^{\infty}J_{n_1}(r)J_{n_2}(r)J_{n_3}(r)J_{n_4}(r)J_{n_5}(r)J_{n_6}(r)rdr.$$
As it turns out, a direct application of Plancherel shows that $||f||_{L^2(\mathbb{S}^1)}^6$ can be expanded in a way that resembles \eqref{FormulaEmanuelDiogo}:
\begin{equation}\label{Norma}
	(2\pi)^{-3}||f||_{L^2(\mathbb{S}^1)}^6=\displaystyle{\sum_{D\in A^3}\sum_{\substack{n_1,n_2,n_3\in A\\n_1+n_2+n_3=D}} |\widehat{f}(n_1)|^2|\widehat{f}(n_2) |^2|\widehat{f}(n_3)|^2}.
\end{equation}
If a good decomposition of the set $A^3$ is available, a direct comparison between \eqref{FormulaEmanuelDiogo} and \eqref{Norma} is possible, which is precisely the idea behind \cite{Ciccone2024}. They prove that if $A$ is a \text{P}(3)-set, among all $L^2$ functions $f$ with $\text{spec}(f)\subset A$, the ratio $||\ft{f\sigma}||_{L^6(\r^2)}/||f||_{L^2(\mathbb{S}^1)}$ is maximized at constant functions. For the sake of completeness, here we recall the definition of a \text{P}(3)-set.
\begin{definition}\label{DefP(3)set}
	We say that a set $A\subset \z$ is a \text{P}(3)-set if for every $D\in A^3$, one and only one of the following holds:
	\begin{enumerate}
		\item[(i)] There exists a unique triple (modulo permutation) $\{n_1,n_2,n_3\}\subset A$ such that $D=n_1+n_2+n_3$;
		\item[(ii)] $D$ has trivial representation, meaning $D=D+m-m$, where $m \in A\cap (-A)$.
	\end{enumerate}
\end{definition}

Given $A\subset \z$, we say that $D\in A^3$ satisfies the $\text{P}(3)$ property (or is a $\text{P}(3)$ point) if $D$ verifies the definition above, possibly even in the case where $A$ is not a $\text{P}(3)$-set itself. Otherwise, we call $D$ an exception. If $D \in A^3$ is a  $\text{P}(3)$ point we say that $D$ is \textit{unique} if it satisfies (i) or  \textit{trivial} if it satisfies (ii).

\subsection{Further Difficulties} One of the main difficulties inherent to the problem of lacunary spectrum with lacunary constant $q<3$ is the fact that the symmetric set $A=\{\pm \lambda_{n}\}$ will no longer be an $\text{P}(2)$-set.
\begin{definition}\label{DefP(2)set}
	A set $A\subset \z$ is said to be an $\text{P}(2)$-set if for every $D\in A^2$, one and only one of the following holds:
	\begin{enumerate}
		\item[(i)] Either there exists a unique pair (modulo permutation) $\{n_1,n_2\}\subset A$ such that $D=n_1+n_2$;
		\item[(ii)] Or $D=0$, which means $D=m-m$ for every $m \in A\cap (-A)$.
	\end{enumerate}
\end{definition}
As in the $\text{P}(3)$ property, we also say that a point $D\in A^2$ satisfying the definition is a $\text{P}(2)$ point, or that $D$ satisfies the $\text{P}(2)$ property, and points in $A^2$ that do not satisfy the Definition \ref{DefP(2)set} are called \textit{exceptions} for the $\text{P}(2)$ property.

The connection between $\text{P}(2)$ and $\text{P}(3)$ sets which it is worth emphasizing is the following: in a symmetric set $A$, every exception for the $\text{P}(2)$ property gives rise to a sort of mixed type exception for the $\text{P}(3)$ property of form $D+m-m=n_1+n_2+n_3$ with $n_i+n_j\neq 0$, $i\neq j$ and arbitrary $m\in A$. In other words, in symmetric sets which are not P(2) sets, there exist exceptions for the \text{P}(3) property with trivial representations, which makes it harder to compare \eqref{FormulaEmanuelDiogo} with \eqref{Norma} in a satisfactory way. Even the case $\text{spec}(f)=\{\pm 3^n\}\cup \{0\}$ has been out of reach since these mixed exceptions are present as in $3+3^n-3^n=1+1+1=9-3-3$. A feature that can be helpful is that, modulo multiplication by $\pm 3^m$, in this case that is the only mixed exception for the \text{P}(3) property.

The most important information about the set $\text{spec}(f)$ that is required is how many solutions an equation as $D=n_1+n_2+n_3$ may have, where $n_i \in \text{spec}(f)$; in the set up $\text{spec}(f)=\{\pm \lambda_{n}\}$, where $\{\lambda_{n}\}$ is lacunary, the number of such solutions increases as the lacunary constant $q$ decreases, in order to go below the barrier $q=3$, a classification result as Lemma \ref{P3exceptions} is not the issue, instead the issue has been to numerically evaluate the integrals $I(n_1,\dots,n_6)$ for small values of $n_1,\dots,n_6$ and still obtain bounds which are small in comparison to the coefficients that naturally appear in the sum \eqref{FormulaEmanuelDiogo}, which is unfeasible so new ideas are needed to tackle the problem with lacunary spectrum with constant $1+\varepsilon$ for $\varepsilon$ small.


\section{Preliminary Results}

From now on we set $A_{\lambda,q}:=\{\pm \lambda_{n}\}$, where $\{\lambda_n\}\subset \z_+=\{0,1,2,\dots\}$ is a lacunary sequence satisfying $\lambda_0=0$ and $\lambda_{n+1}/\lambda_{n}>q\geq 3$ for every $n\geq 1$. We can focus on infinite sequences $\{\lambda_n\}$ since the finite case would easily follow.

As explained in the introduction, the key part in the argument is to have in hands a good decomposition of $A_{\lambda,q}^3$. More specifically, we are interested in information about the number of solutions of the equation $D=n_1+n_2+n_3$, where  $D\in A_{\lambda,q}^3$ and $n_1,n_2,n_3\in A_{\lambda,q}$. As the number $q$ increases, the more sparse the sequence $A_{\lambda,q}$ is and better behaved the set $A^3_{\lambda,q}$ becomes. For the sake of simplicity let $A_{\lambda,q}^3=A_{\text{P(3)}}^3\cup A_{E}^3$, where $A_{\text{P}(3)}^3:=\{ D\in A_{\lambda,q}^3;\, \text{$D$ is a P(3) point}\}$ and $A_{E}^3:=\{ D\in A_{\lambda,q}^3;\, \text{$D$ is not a P(3) point}\}$. A full description of $A_E^3$ is given by the next lemma.
\begin{lemma}\label{P3exceptions}
	Let $A_{\lambda,q}=\{\pm \lambda_{n}\}$. Then $A_{E}^3\neq \varnothing$ if and only if one of the equations
	\begin{equation}\label{DescriptionOfExceptions}
		\begin{array}{rlr}
			\lambda_{n+1}&=\lambda_{n}+\lambda_{n}+\lambda_{n}+\lambda_{m}+\lambda_{k},& 0\leq \lambda_{m}\leq \lambda_k\leq \lambda_{n},\lambda_{k}>0,\\
			\lambda_{n+1}+\lambda_{m}&=\lambda_{n}+\lambda_{n}+\lambda_{n}+\lambda_{k},&0\leq \lambda_{m}<\lambda_{k}\leq \lambda_{n},
		\end{array}
	\end{equation}
	has a solution. Furthermore, if $D \in A^3_E$ then $D$ is expressed in one of the following ways:
	\begin{enumerate}
		\item[(i)] $ D=\pm (\lambda_{n+1}-\lambda_{n}-\lambda_{n})=\pm(\lambda_{n}+\lambda_{m}+\lambda_{k})$, where $0\leq \lambda_{m}\leq \lambda_{k}\leq \lambda_{n}$ and $\lambda_{k}>0$,
		\item[(ii)] $ D=\pm(\lambda_{n+1}-\lambda_{n}-\lambda_{m})=\pm(\lambda_{n}+\lambda_{n}+\lambda_{k})$, where $0\leq \lambda_{m}\leq \lambda_{k}\leq \lambda_{n}$ and $\lambda_{k}>0$,
		\item[(iii)] $ D=\pm(\lambda_{n+1}-\lambda_{n}-\lambda_{k})=\pm(\lambda_{n}+\lambda_{n}+\lambda_{m})$, where $0\leq \lambda_{m}\leq \lambda_{k}\leq \lambda_{n}$ and $\lambda_{k}>0$,
		\item[(vi)] $ D=\pm(\lambda_{n+1}-\lambda_{m}-\lambda_{k})=\pm(\lambda_{n}+\lambda_{n}+\lambda_{n})$, where $0\leq \lambda_{m}\leq \lambda_{k}\leq \lambda_{n}$ and $\lambda_{k}>0$,
		\item[(v)] $ D=\pm(\lambda_{n+1}-\lambda_{n}-\lambda_{n})=\pm(\lambda_{n}-\lambda_{m}+\lambda_{k})$, where $0\leq \lambda_{m}< \lambda_{k}\leq \lambda_{n}$,
		\item[(vi)] $ D=\pm(\lambda_{n+1}-\lambda_{n}+\lambda_{m})=\pm(\lambda_{n}+\lambda_{n}+\lambda_{k})$, where $0\leq \lambda_{m}< \lambda_{k}\leq \lambda_{n}$,
		\item[(vii)] $ D=\pm(\lambda_{n+1}-\lambda_{n}-\lambda_{k})=\pm(\lambda_{n}+\lambda_{n}-\lambda_{m})$, where $0\leq \lambda_{m}< \lambda_{k}\leq \lambda_{n}$,
		\item[(viii)] $  D=\pm(\lambda_{n+1}+\lambda_{m}-\lambda_{k})=\pm(\lambda_{n}+\lambda_{n}+\lambda_{n})$, where $0\leq \lambda_{m}< \lambda_{k}\leq \lambda_{n}$.
	\end{enumerate}
	In particular, $D\in A^3_E$ if and only if, the equation $D=n_1+n_2+n_3$ for $n_1,n_2,n_3\in A_{\lambda,q}$ has exactly two different solutions (modulo permutations).
\end{lemma}

\begin{example}
	Before we prove Lemma \ref{P3exceptions}, we illustrate how this lemma can be used to describe \text{P}(3) exceptions in specific examples. Let $A_{4}=\{\pm 4^{n-1};n\geq 1\}\cup \{0\}$ and $A_{5}=\{\pm 5^{n-1};n\geq 1\}\cup \{0\}$. As pointed out in \cite{Ciccone2024}, these sets are not \text{P}(3) sets. However, as the sets $\{4^{n-1};n\geq 1\}$ and $\{5^{n-1};n\geq 1\}$ are lacunary sequences with lacunary constant $>3$, every exception in $A_4^3$ and $A_5^3$ for the \text{P}(3) property must generate solutions for the equations \eqref{DescriptionOfExceptions}. Observe that for the case of $A_5$, equations \eqref{DescriptionOfExceptions} become
	\begin{equation*}
		\begin{array}{rlr}
			5^{n+1}&=5^{n}+5^{n}+5^{n}+\lambda_{m}+5^{k},& 0\leq \lambda_{m}\leq 5^k,0\leq k\leq n,\\
			5^{n+1}+\lambda_{m}&=5^{n}+5^{n}+5^{n}+5^{k},&0\leq \lambda_{m}<5^{k}, 0 \leq k \leq n,
		\end{array}
	\end{equation*}
	where $\lambda_m\in A_5$. It is easy to see that the first equation above has solutions only if $\lambda_m=5^k$ and $k=n$ (otherwise the right hand side would be strictly smaller than the left hand side), whereas the second equation has no solutions. Thus for the set $A_5$, equations \eqref{DescriptionOfExceptions} become simply $5^{n+1}=5^n+5^n+5^n+5^n+5^n$, $n\geq 0$. After performing all possible rearrangements of three terms to each side, we conclude that all exceptions $D\in A_5^3$ for the \text{P}(3) property are expressed as $D=\pm (5^{n+1}-5^n-5^n)=\pm (5^n+5^n+5^n)$. In other words, $\{D\in A^3_5; \text{$D$ is not a $\text{P}(3)$ point}\}=\{\pm 3\cdot 5^n; n\geq 0 \}$, and all exceptions are, modulo multiplication by $\pm 5^n$ for some $n$, written as the simple expression: $3=5-1-1=1+1+1$. Similarly, for the special case $A_4$, equations \eqref{DescriptionOfExceptions} become simply $4^{n+1}=4^n+4^n+4^n+4^n$, which gives us $\{D\in A^3_4; \text{$D$ is not an $\text{P}(3)$ point}\}=\{\pm 2\cdot 4^n; \, n\geq 0 \}\cup\{\pm 3\cdot 4^n; \, n\geq 0 \}$, and every such exception is written as, after a multiplication by $\pm 4^n$ for some $n\geq 0$, of one of the simple expressions: $2=4-1-1=1+1+0$ or $3=4-1+0=1+1+1$.
\end{example}

\begin{proof}[Proof of Lemma \ref*{P3exceptions}]
	Let $A_{\lambda,q}=\{\pm \lambda_n\}$, where $\{\lambda_n\}$ is a lacunary sequence with lacunary constant $q\geq 3$. Our goal is to find a description of when an equation such as
	\begin{equation}\label{IgualdadeA^3_E}
		n_1+n_2+n_3=n_4+n_5+n_6,\, n_i\in A_{\lambda,q},
	\end{equation}
	can be satisfied. To avoid trivialities, suppose that at least one $n_i$ is non-zero. We start by reorganizing the terms so we get only positive numbers in each side of the equation, i.e.,
	\begin{equation*}
		0<\lambda_{\alpha_1}+\dots+\lambda_{\alpha_k}=\lambda_{\beta_1}+\dots+\lambda_{\beta_l}
	\end{equation*}
	where $k+l\leq 6$, with $k,l\geq 1$ and $\{\lambda_{\alpha_1},\dots,\lambda_{\alpha_k},\lambda_{\beta_1},\dots,\lambda_{\beta_l}\}\subset \{|n_1|,\dots,|n_6|\}$. If $\max\{\lambda_{\alpha_1},\dots,\lambda_{\alpha_k}\}=\max\{\lambda_{\beta_1},\dots,\lambda_{\beta_l}\}$, it follows from the fact that $A_{\lambda,q}$ is an P(2)-set that $\{\lambda_{\alpha_1},\dots,\lambda_{\alpha_k}\}=\{\lambda_{\beta_1},\dots,\lambda_{\beta_l}\}$, which means that $D:=n_1+n_2+n_3$ is a \text{P}(3) point. Now if $\max\{\lambda_{\alpha_1},\dots,\lambda_{\alpha_k}\}\neq \max\{\lambda_{\beta_1},\dots,\lambda_{\beta_l}\}$, say $\lambda_{n+1}:=\max\{\lambda_{\alpha_1},\dots,\lambda_{\alpha_k}\}>\max\{\lambda_{\beta_1},\dots,\lambda_{\beta_l}\}=:\lambda_{s}$, we must have
	$$3\lambda_s<\lambda_{n+1}\leq \lambda_{\beta_1}+\dots+\lambda_{\beta_l}\leq 5 \lambda_{s}<3^2\lambda_{s},$$
	which implies $s=n$. In particular, we have (relabeling the terms if necessary):
	\begin{equation*}
		\lambda_{\alpha_1}+\dots+\lambda_{\alpha_{k-1}}+\lambda_{n+1}=\lambda_{\beta_1}+\dots+\lambda_{\beta_{l-1}}+\lambda_n,
	\end{equation*}
	which is possible only if
	$$\lambda_{n+1}=\lambda_n+\lambda_n+\lambda_n+\lambda_m+\lambda_k$$
	where $0\leq \lambda_m\leq \lambda_k\leq \lambda_{n}, \lambda_m+\lambda_k>0$,
	or
	$$\lambda_{n+1}+\lambda_m=\lambda_n+\lambda_n+\lambda_n+\lambda_k,$$
	where $ 0\leq \lambda_m<\lambda_k\leq \lambda_{n}$. Reorganizing the terms back in the form \eqref{IgualdadeA^3_E} one get that $D:=n_1+n_2+n_3\in A^3_E$. Moreover, all possible rearrangements of tree terms to each side give us the complete description of points in $A^3_E$.
\end{proof}

\begin{remark}\label{Nodoublerepetition}
	The fact that each $D\in A^3_E$ is written in exactly 2 different ways is an easy consequence of the fact that for each such $D$, say $D>0$, there exists an $n>0$ such that $\lambda_n<D<\lambda_{n+1}$, and now the claim follows from the observation that $A_{\lambda,q}$ is an P(2) set. It also follows from Lemma \ref{P3exceptions} that if $D\in A^3_{E}$, then at least one of the two different triples that sum up to $D$ has one element which repeats itself, i.e., every such $D$ is written as $n_1+n_2+n_3=m_1+m_1+m_2=D$. This is simply due to the fact that there always is an element in equations \eqref{DescriptionOfExceptions} that repeats itself at least 3 times. This innocent remark will be of importance later.
\end{remark}

Now we state the last piece of information needed to prove our main result. We need to improve upon \cite[Lemma 2]{Ciccone2024} to the case that best suits our purposes. Let $F:(\z_+)^3\to \r_+$ be the function
$$F(n_1,n_2,n_3)=\dfrac{\I(0,0,0)}{\I(n_1,n_2,n_3)}, \text{ where $\I(n_1,n_2,n_3):= I(n_1,n_1,n_2,n_2,n_3,n_3)$.}$$
From known asymptotics for Bessel functions, Ciccone and Gonçalves obtained explicit upper bounds for the growth of $\I(n_1,n_2,n_3)$ in terms of $\max\{n_1,n_2,n_3\}$, obtaining that $\I(n_1,n_2,n_3)$ can be upper bounded by a decreasing function of $\max\{n_1,n_2,n_3\}$ that decays as $O(\max\{n_1,n_2,n_3\}^{-\frac{1}{3}})$. In other words, $F(n_1,n_2,n_3)$ goes to infinity at least as fast as an increasing function that grows as $\gtrsim\max\{n_1,n_2,n_3\}^{\frac{1}{3}}$. Our task is reduced to check that $F(n_1,n_2,n_3)$ is not that small for small values of $n_i$ as well. More precisely, we need the following lemma, which is essentially due to \cite{Ciccone2024}.
\begin{lemma}\label{BoundsforF}
	The following lower bounds holds:
	\begin{enumerate}
		\item [(a)] For all $n\geq 1$, $F(n,0,0)\geq 5$ with equality only when $n=1$.
		\item [(b)] For all $n\geq 1$, $F(n,n,0)> 7.94$. Moreover, $F(n,n,0)> 10.8$ for $n\geq 3$.
		\item [(c)]For all $n\geq 1$, $F(n,n,n)> 3.2$.
		\item [(d)] For all $n,m\geq 1$ with $n\neq m$ and $(n,m)\neq (1,2)$, $F(n,n,m)>10$. Moreover, if $n,m\in \{\lambda_n\}$ then $F(n,n,m)>13.2$.
		\item [(e)] For all $n>m>k\geq 0$ with $(n,m,k)\neq (3,2,0)$, $F(n,m,k)>18$. Moreover, if $n,m,k\in \{\lambda_n\}$ then $F(n,m,k)>21$.
	\end{enumerate}
\end{lemma}

The idea to prove Lemma \ref{BoundsforF} is simple: we use the explicit asymptotic available in \cite{Ciccone2024} to verify that after some fixed value of $\max\{n_1,n_2,n_3\}$ each bound in Lemma \ref{BoundsforF} holds, and then numerically we verify the desired bound holds for the remainders values. To perform numerical analysis faster, we use the following modified version of \cite[Lemma 3]{Ciccone2024}.
\begin{lemma}\label{Lema8}
	Let $k,m,n\geq 0$ be integers such that $\max\{k,m,n\}\leq 532$. Then
	\begin{equation*}
		0<\I(k,m,n)- \widetilde{\I}(k,m,n)<10^{-2},
	\end{equation*}
	where
	$$\widetilde{\I}(k,m,n):=\dfrac{2}{9}\sum_{r=0}^{1000} \dfrac{J_k^2 (\sigma_r/3)J_m^2 (\sigma_r/3)J_n^2 (\sigma_r/3)}{J_0^2 (\sigma_r)},$$
	and the sequence $\{\sigma_r\}_{r\geq 0}$ denotes the non-negative zeros of the Bessel function $J_1$.
\end{lemma}

The proof of Lemma \ref{Lema8} is straightforward modification of \cite[Lemma 3]{Ciccone2024}. We leave the details for the interested reader.

In what follows, the numerical calculations were performed using [PARI-GP, version 2.15.3] computer algebra system.

\begin{proof}[Proof of Lemma \ref{BoundsforF}]
	From the explicit bounds for $F(n_1,n_2,n_3)$ given in \cite{Ciccone2024}, we have that $F(n,n,0)>10.8$ for every $n\geq \text{21}$. If $3\leq n <21$ we use Lemma \ref{Lema8} to check that we have $F(n,n,0)>10.8$. Similarly from the explicit bounds, we have that $F(n,n,m)>13.2$ whenever $n,m\geq 1$ and $\max\{n,m\}\geq 110$, now if $n,m\in \{\lambda_n\}\backslash\{0\}$ with $n\neq m$, then we can check numerically that $F(n,n,m)>13.2$ for $\max\{n,m\}<110$. Again from the explicit bounds for $F$, whenever $0 \leq k<m<n$ and $n\geq 340$, $F(k,m,n)>21$, whereas if $n<340$, $F(k,m,n)>21$.
	
	To conclude the proof of our lemma, we remark that all of the remaining inequalities follow from checking more precisely the constants in \cite[Lemma 2]{Ciccone2024}.
\end{proof}

\section{Proof of Main Result}

\begin{proof}[Proof of Theorem \ref{Lacunarycaseq>3}]
The starting point is the formula \eqref{FormulaEmanuelDiogo}. Consider the decomposition $A_{\lambda,q}^3=A_{\text{P(3)}}^3\cup A_{E}^3$, we write
\begin{align*}
	\begin{array}{rl}
		&S:=(2\pi)^{-7}||\widehat{f\sigma}||^{6}_{L^6(\r^2)}=\\
		&=\displaystyle{\left(\sum_{D\in A_{\text{P(3)}}^3}+\sum_{D\in A_E^3}\right)\sum_{\substack{n_1,\dots,n_6\in A\\n_1+n_2+n_3=D\\n_4+n_5+n_6=D}} \widehat{f}(n_1)\widehat{f}(n_2) \widehat{f}(n_3)\overline{\widehat{f}(n_4)}\overline{\widehat{f}(n_5)}\overline{\widehat{f}(n_6)}I(n_1,\dots,n_6)}\\
		&=S_{\text{P(3)}}+S_E.
	\end{array}
\end{align*}

Now our goal is to find upper bounds for $S_{\text{P(3)}}$ and $S_E$. Let us deal with $S_E$. To simplify notation, for each $ D\in A^3_E$ let $\{n_1^D,n_2^D,n_3^D\}$ and $\{n_4^D,n_5^D,n_6^D\}$ be the two different triples that satisfy $n_1+n_2+n_3=D$, and let $p(n_1,n_2,n_3)$ denote the number of permutations of $\{n_1,n_2,n_3\}$. Thus it follows that
	\begin{align*}
		&S_E=\displaystyle{\sum_{D\in A_E^3}\sum_{\substack{n_1,\cdots,n_6\in A_{\lambda,q}\\n_1+n_2+n_3=D\\n_4+n_5+n_6=D}} \widehat{f}(n_1)\widehat{f}(n_2) \widehat{f}(n_3)\overline{\widehat{f}(n_4)}\overline{\widehat{f}(n_5)}\overline{\widehat{f}(n_6)}I(n_1,\cdots,n_6)}\\
		&=\displaystyle{\sum_{D\in A_E^3} p(n_1^D,n_2^D,n_3^D)^2 |\widehat{f}(n_1^D)|^2|\widehat{f}(n_2^D)|^2 |\widehat{f}(n_3^D)|^2\I(n_1^D,n_2^D,n_3^D)}\\
		&+\displaystyle{\sum_{D\in A_E^3} p(n_4^D,n_5^D,n_6^D)^2 |\widehat{f}(n_4^D)|^2|\widehat{f}(n_5^D)|^2 |\widehat{f}(n_6^D)|^2\I(n_4^D,n_5^D,n_6^D)}\\
		&+\displaystyle{2\re \sum_{D\in A_E^3} p(n_1^D,n_2^D,n_3^D)p(n_4^D,n_5^D,n_6^D)\widehat{f}(n_1^D)\widehat{f}(n_2^D) \widehat{f}(n_3^D)\overline{\widehat{f}(n_4^D)}\overline{\widehat{f}(n_5^D)}\overline{\widehat{f}(n_6^D)}I(n_1^D,\cdots,n_6^D)}.
	\end{align*}
	
	From the Cauchy-Schwarz inequality, we get $I(n_1,\dots,n_6)\leq \I(n_1,n_2,n_3)^{\frac{1}{2}}\I(n_4,n_5,n_6)^{\frac{1}{2}}$. Now it follows from the triangular inequality, the basic inequality $2xy\leq \dfrac{x^2}{\varepsilon}+\varepsilon y^2$ and the identity $J_{-n}=(-1)^nJ_n$ that we have the bound
	\begin{align}\label{BoundforSe}
		S_E&\leq \displaystyle{\sum_{D\in A_E^3}\left(1+\dfrac{1}{\varepsilon_D}\right) p(n_1^D,n_2^D,n_3^D)^2 |\widehat{f}(n_1^D)|^2|\widehat{f}(n_2^D)|^2 |\widehat{f}(n_3^D)|^2\I(n_1^D,n_2^D,n_3^D)}\\
		&+\displaystyle{\sum_{D\in A_E^3}\left(1+\varepsilon_D\right) p(n_4^D,n_5^D,n_6^D)^2 |\widehat{f}(n_4^D)|^2|\widehat{f}(n_5^D)|^2 |\widehat{f}(n_6^D)|^2\I(n_4^D,n_5^D,n_6^D)},
	\end{align}
	where $\varepsilon_D$ is a positive number that is to be chosen later depending on the two different triples that satisfies $D=n_1+n_2+n_3$. Let us further decompose the set $A^3_E$ as follows: let $A_1^3$ be the set of points $D\in A_E^3$ such that the two different triples that sum to $D$ have repeated elements, and $A^3_2:=A^3_E\backslash A^3_1$, i.e., $A_2^3$ denotes the subset of points $D\in A^3_E$ such that one, and by Remark \ref{Nodoublerepetition} only one, of the two triples whose sum is $D$ has no repeated elements. The upper bound \eqref{BoundforSe} may be rewritten as
	\begin{align*}
		&S_E\leq \sum_{\substack{n_1,n_2,n_3\in A_{\lambda,q}\\|n_i|\neq |n_j|, i\neq j}}6\left(1+\dfrac{1}{\varepsilon_D}\right)\delta_{n_1+n_2+n_3\in A_2^3} |\widehat{f}(n_1)|^2|\widehat{f}(n_2)|^2|\widehat{f}(n_3)|^2\I(n_1,n_2,n_3)\\
		&+\sum_{\substack{n_1\in A_{\lambda,q}\backslash \{0\},n_2\in A_{\lambda,q}\\|n_1|\neq |n_2|}}9\left(\delta_{2n_1+n_2\in A^3_E}+{\varepsilon_D}\delta_{2n_1+n_2\in A_2^3} \right)  |\widehat{f}(n_1)|^4|\widehat{f}(n_2)|^2\I(n_1,n_1,n_2)\\
		&+\sum_{\substack{n_1\in A_{\lambda,q}\backslash \{0\},n_2\in A_{\lambda,q}\\|n_1|\neq |n_2|}}9\left(\dfrac{\delta_{2n_1+n_2\in A^3_1\cap 3A_{\lambda,q}}}{\varepsilon_D}+\varepsilon_D^a\delta_{2n_1+n_2\in A^3_1\backslash 3A_{\lambda,q}} \right)  |\widehat{f}(n_1)|^4|\widehat{f}(n_2)|^2\I(n_1,n_1,n_2)\\
		&+\sum_{\substack{n_1\in A_{\lambda,q}\backslash \{0\}}}  \left(\delta_{3n_1\in A^3_2}+\varepsilon_D\delta_{3n_1\in A^3_1}+{\varepsilon_D}\delta_{3n_1\in A_2^3} \right) |\widehat{f}(n_1)|^6 \I(n_1,n_1,n_1),
	\end{align*}
	where $a\in \{-1,0,1\}$ is chosen to be zero if $D\notin 2A_{\lambda,q}$, and if $D=2n_1+n_2\in 2A_{\lambda,q}$, we set $a=1$ if $n_2=0$ and $a=-1$ if $n_2\neq 0$. This convention is well define since each $D\in A_1^3\backslash 3A_{\lambda,q}$ is written as $D=2n_1+n_2=2m_1+m_2$ with $\{n_1,n_2\}\neq \{m_1,m_2\}$ and $n_2,m_2$ can not simultaneously be zero.
	
	Now to find a satisfactory upper bound for $S_{\text{P(3)}}$ we can perform the same technique as in \cite{Ciccone2024} to control the sum over the trivial points: as a key step in their proof, they used the basic inequality $r^3s\leq \frac{5}{8}r^4+\frac{1}{8}s^4+\frac{1}{4}r^2s^2$. Instead we use the basic inequality
	\begin{equation}\label{FamillyofBasicinequalities}
		r^3s\leq \dfrac{b}{2b-2}r^4+\frac{1}{2b-2}s^4+\frac{b-3}{2b-2}r^2s^2, \, \forall b>1,
	\end{equation}
	and we get the same bound for $S_{\text{P}(3)}$ as in \cite{Ciccone2024} with a slight modification in the coefficients that now depend on $b$, with the advantage that now we can optimize the parameter $b$ (note that letting $b=5$ we recover $r^3s\leq\frac{5}{8}r^4+\frac{1}{8}s^4+\frac{1}{4}r^2s^2$). Now, putting together the upper bound for $S_E$ with the upper bound for $S_{\text{P}(3)}$ we have
	\begin{align*}
		&S \leq  \sum_{\substack{n_1,n_2,n_3\in A_{\lambda,q}\\|n_i|\neq |n_j|, i\neq j}}\left(15+\dfrac{6}{\varepsilon_D}\delta_{n_1+n_2+n_3\in A_2^3}\right) |\widehat{f}(n_1)|^2|\widehat{f}(n_2)|^2|\widehat{f}(n_3)|^2\I(n_1,n_2,n_3)\\
		&+\sum_{\substack{n_1\in A_{\lambda,q}\backslash \{0\},n_2\in A_{\lambda,q}\\|n_1|\neq |n_2|}}9\left(1+\delta_{n_1,n_2\in A_{\lambda,q}\backslash\{0\}}+{\varepsilon_D}\delta_{2n_1+n_2\in A_2^3} \right)  |\widehat{f}(n_1)|^4|\widehat{f}(n_2)|^2\I(n_1,n_1,n_2)\\
		&+\sum_{\substack{n_1\in A_{\lambda,q}\backslash \{0\},n_2\in A_{\lambda,q}\\|n_1|\neq |n_2|}}9\left(\dfrac{\delta_{2n_1+n_2\in A^3_1\cap 3A_{\lambda,q}}}{\varepsilon_D}+\varepsilon_D^a\delta_{2n_1+n_2\in A^3_1\backslash 3A_{\lambda,q}} \right)  |\widehat{f}(n_1)|^4|\widehat{f}(n_2)|^2\I(n_1,n_1,n_2)\\
		&+\sum_{\substack{n_1\in A_{\lambda,q}\backslash \{0\}}}  \left(1+\varepsilon_D\delta_{3n_1\in A^3_1}+{\varepsilon_D}\delta_{3n_1\in A_2^3} \right) |\widehat{f}(n_1)|^6 \I(n_1,n_1,n_1)\\
		&+\sum_{\substack{n_1,n_2\in A_{\lambda,q} \\ |n_1|\neq |n_2|}}(18-9\delta_{n_2=0}) |\widehat{f}(n_1)|^2|\widehat{f}(n_2)|^2|\widehat{f}(-n_2)|^2\I(n_1,n_2,n_2)\\
		&+\sum_{\substack{n_1,n_2\in A_{\lambda,q}\backslash\{0\}\\ |n_1|\neq |n_2|}} 9|\widehat{f}(n_1)|^2|\widehat{f}(-n_1)|^2|\widehat{f}(n_2)|^2\I(n_1,n_1,n_2)\\
		&+\dfrac{9(b+1)}{b-1}\sum_{n_1\in A_{\lambda,q}\backslash\{0\}}|\widehat{f}(n_1)|^4|\widehat{f}(0)|^2\I(n_1,n_1,0)\\
		&+\dfrac{9(b-3)}{b-1}\sum_{n_1\in A_{\lambda,q}\backslash\{0\} }|\widehat{f}(n_1)|^2|\widehat{f}(-n_1)|^{2}|\widehat{f}(0)|^2\I(n_1,n_1,0)+6\sum_{n_1\in A_{\lambda,q}\backslash\{0\}}|\widehat{f}(n_1)|^2|\widehat{f}(0)|^{4}\I(n_1,0,0)\\&
		+9\sum_{n_1\in A_{\lambda,q}\backslash\{0\}}|\widehat{f}(n_1)|^4|\widehat{f}(-n_1)|^{2}\I(n_1,n_1,n_1)+|\widehat{f}(0)|^6\I(0,0,0),
	\end{align*}
	for every $b> 1$.
	
	Similarly we can write $||f||_{L^2(\mathbb{S}^1)}^6$ as
	\begin{align*}
		(2\pi)^{-3}||f||_{L^2(\mathbb{S}^1)}^6&= \sum_{\substack{n_1,n_2,n_3\in A_{\lambda,q}\\|n_i|\neq |n_j|, i\neq j }}|\widehat{f}(n_1)|^2|\widehat{f}(n_2)|^2|\widehat{f}(n_3)|^2+\sum_{\substack{n_1,n_2\in A_{\lambda,q}\backslash\{0\}\\ |n_1|\neq |n_2|}}3 |\widehat{f}(n_1)|^4|\widehat{f}(n_2)|^2\\
		&+\sum_{n_1\in A_{\lambda,q}\backslash\{0\}} |\widehat{f}(n_1)|^6+\sum_{\substack{n_1,n_2\in A_{\lambda,q}\backslash\{0\}\\ |n_1|\neq |n_2|}}3 |\widehat{f}(n_1)|^2|\widehat{f}(n_2)|^2|\widehat{f}(-n_2)|^2\\
		&+\sum_{n_1\in A_{\lambda,q}\backslash\{0\}} 3|\widehat{f}(n_1)|^4|\widehat{f}(0)|^2+\sum_{n_1\in A_{\lambda,q}\backslash\{0\}} 3|\widehat{f}(n_1)|^2|\widehat{f}(-n_1)|^2|\widehat{f}(0)|^2\\
		&+\sum_{n_1\in A_{\lambda,q}\backslash\{0\}} 3|\widehat{f}(n_1)|^2|\widehat{f}(0)|^4+\sum_{n_1\in A_{\lambda,q}\backslash\{0\}} 3|\widehat{f}(n_1)|^4|\widehat{f}(-n_1)|^2+|\widehat{f}(0)|^6.
	\end{align*}
	
	To prove the inequality $||\widehat{f\sigma}||^6_{L^6(\r^2)}\leq (2\pi)^4\left(\int_0^{\infty}J_0^6(r)rdr\right)||f||^6_{L^2(\mathbb{S}^1)}$ it suffices to verify that there exists $\{\varepsilon_D\}_{D\in A^3_E}$ and $b>1$ such that the following inequalities are true:
		\begin{equation}\label{Sistematrivial}
		\left\{\begin{array}{l}
			9\leq 3F(n_1,n_1,n_1),|n_1|>0\\
			15\leq 3 F(n_1,0,0),|n_1|>0\\
			\dfrac{9(b-3)}{b-1}+18\leq 3F(n_1,n_1,0),|n_1|>0\\
			\dfrac{9(b+1)}{b-1}+9\leq 3F(n_1,n_1,0),|n_1|>0\\
			27 \leq 3F(m_1,m_1,m_2),|m_1|\neq |m_2|, |m_1|,|m_2|>0\\
			15\leq F(n_1,n_2,n_3),|n_1|>|n_2|>|n_3|\geq 0
		\end{array}\right.
	\end{equation}
	and also that there exist $\varepsilon_D>0$ such that the following systems of inequalities have solutions:
	\begin{equation}\label{System2}
		\left\{\begin{array}{l}
			15+\dfrac{6}{\varepsilon_D} \leq F(n_1,n_2,n_3), |n_1|>|n_2|>|n_3|> 0\\
			\dfrac{9(b+1)}{b-1}\delta_{m_2=0}+9(1+\delta_{m_1,m_2\in A_{\lambda,q}\backslash\{0\}}+\varepsilon_D)\leq 3F(m_1,m_1,m_2),|m_1|\neq |m_2|,|m_1|>0
		\end{array}\right.
	\end{equation}
	where $n_1+n_2+n_3=2m_1+m_2$,
		\begin{equation}\label{System3}
		\left\{\begin{array}{l}
			15+\dfrac{6}{\varepsilon_D} \leq F(n_1,n_2,n_3), |n_1|>|n_2|>|n_3|> 0\\
			1+\varepsilon_D\leq F(m_1,m_1,m_1),|m_1|>0
		\end{array}\right.
	\end{equation}
	where $n_1+n_2+n_3=3m_1$, also
		\begin{equation}\label{System4}
		\left\{\begin{array}{l}
			\dfrac{9(b+1)}{b-1}+9(1+\varepsilon_D)\leq 3F(n_1,n_1,0),|n_1|>0\\
			9\left(2+\dfrac{1}{\varepsilon_D}\right)\leq 3F(m_1,m_1,m_2),|m_1|\neq |m_2|,|m_2|,|m_1|>0
		\end{array}\right.
		\end{equation}
	where $2n_1=2m_1+m_2$ and finally
	\begin{equation}\label{System5}
		\left\{\begin{array}{l}
			9\left(2+\dfrac{1}{\varepsilon_D}\right)\leq 3F(n_1,n_1,n_2),|n_1|\neq |n_2|,|n_2|,|n_1|>0\\
			1+\varepsilon_D\leq F(m_1,m_1,m_1),|m_1|>0
		\end{array}\right.
	\end{equation}
	where $2n_1+n_2=3m_1$.
	
From Lemma \ref{BoundsforF} it is clear that all inequalities in \eqref{Sistematrivial} are true as long as $4.2 \leq b\leq 6.6604...$, and that the systems \eqref{System3} and \eqref{System5} have solutions if one takes $\varepsilon_D$ to be 2 and 1, respectively. From now on we can fix $b=6.66$. If $|n_1|\geq 3$ in \eqref{System4}, then we also have solutions by simply taking $\varepsilon_D=1$; if $|n_1|\in \{1,2\}$, in conjunction with Lemma \ref{P3exceptions} we can check numerically that we also have solutions for \eqref{System4}. In fact, if $|n_1|=1$, then by Lemma \ref{P3exceptions} we know that $\{|m_1|,|m_2|\}=\{1,4\}$, thus $F(n_1,n_1,0)>7.94$ whereas numerically we have $F(m_1,m_1,m_2)>17.3$, which allows us to take any $0.270<\varepsilon_D<0.289$. Similarly, if $|n_2|=2$ we have $\{|m_1|,|m_2|\}=\{2,8\}$ and one can take $0.110<\varepsilon_D<0.490$. To handle \eqref{System2}, first observe that if $m_2\neq 0$, then by Lemma \ref{BoundsforF} we see that one can take $\varepsilon_D=2$. If $m_2=0$, then we see that $\varepsilon_D$ can be taken to be $\varepsilon_D=1$, because $F(n_1,n_2,n_3)>21$ and the condition $n_1+n_2+n_3=m_1+m_1+0$ with $|n_1|>|n_2|>|n_3|\geq 0$ implies by Lemma \ref{P3exceptions} that $(|n_1|,|n_2|,|n_3|)=(3|m_1|+\lambda_k,|m_1|, \lambda_k)$, where $0<\lambda_{k}<|m_1|$. In particular, we have $|m_1|>3$, which by Lemma \ref{BoundsforF} implies that $F(m_1,m_1,0)>10.8$. Thus the proof of \eqref{SharpInequality} is complete. As every inequality in Lemma \ref{BoundsforF} was strict, we have equality in \eqref{SharpInequality} only if $\text{spec}(f)=\{0\}$, which concludes the proof of Theorem \ref{Lacunarycaseq>3}.
\end{proof}

\section*{Acknowledgments}
The authors are grateful to David José Melo, Antonio Pedro Ramos and João P. G. Ramos for stimulating discussions during the preparation of this work. The first author acknowledges support from the following funding agencies: The Office of Naval Research GRANT14201749 (award number N629092412126), The Serrapilheira Institute (Serra-2211-41824), FAPERJ (E-26/200.209/2023) and CNPq (309910/2023-4). The second author acknowledges the support of CAPES (88887.949809/2024-00) scholarship.

	\bibliographystyle{abbrv}
	\bibliography{References}
\nocite{GPpari}

\end{document}